\documentclass[11pt]{amsart}

\usepackage{geometry}
\usepackage{todonotes}

\usepackage{amsfonts, amssymb, amscd}
\numberwithin{equation}{section}

\usepackage[symbol]{footmisc}

\usepackage{bm}
\usepackage{verbatim}
\usepackage{mathrsfs}
\usepackage{graphicx}
\usepackage{tikz-cd}
\usepackage{subcaption}
\usepackage{listings}
\usepackage{subfiles}
\usepackage[toc,page]{appendix}
\usepackage{mathtools}
\usepackage{comment}
\usepackage{enumerate}
\usepackage{enumitem}

\usepackage{graphicx}
\graphicspath{{images/}}

\usepackage{appendix}
\usepackage{hyperref}
\newcommand{\Qq}{\mathbb{Q}}

\newcommand{\Rr}{\mathbb{R}}

\newcommand{\Center}{\operatorname{center}}

\newcommand{\Exc}{\operatorname{Exc}}

\newcommand{\mld}{{\rm{mld}}}

\newcommand{\Supp}{\operatorname{Supp}}

\newcommand{\mult}{\operatorname{mult}}

\newcommand{\Dd}{\mathcal{D}}

\newtheorem{thm}{Theorem}[section]

\newtheorem{lem}[thm]{Lemma}

\newtheorem{claim}[thm]{Claim}

\theoremstyle{definition}
\newtheorem{defn}[thm]{Definition}
\newtheorem{ques}[thm]{Question}
\theoremstyle{definition}
\newtheorem{rem}[thm]{Remark}

\newtheorem{deflem}[thm]{Definition-Lemma}
\newtheorem{ex}[thm]{Example}

\theoremstyle{definition}

\begin{document}

\title{Divisors computing minimal log discrepancies on lc surfaces}
\author{Jihao Liu and Lingyao Xie}

\address{Department of Mathematics, The University of Utah, Salt Lake City, UT 84112, USA}
\email{lingyao@math.utah.edu}

\address{Department of Mathematics, The University of Utah, Salt Lake City, UT 84112, USA}
\email{jliu@math.utah.edu}

\date{\today}

\begin{abstract}
Let $(X\ni x,B)$ be an lc surface germ. If $X\ni x$ is klt, we show that there exists a divisor computing the minimal log discrepancy of $(X\ni x,B)$ that is a Koll\'ar component of $X\ni x$. If $B\not=0$ or $X\ni x$ is not Du Val, we show that any divisor computing the minimal log discrepancy of $(X\ni x,B)$ is a potential lc place of $X\ni x$.
\end{abstract}

\maketitle

\tableofcontents

\section{Introduction}

The \emph{minimal log discrepancy (mld)} is an invariant that provides a sophisticated measure of the singularities of an algebraic variety. It not only plays an important role in the study of singularities but is also a central object in the minimal model program. Shokurov proved that the ascending chain condition (ACC) conjecture for mlds and the lower-semicontinuity (LSC) conjecture for mlds imply the termination of flips \cite{Sho04}. For papers related to these conjectures, we refer the readers to \cite{Sho96,Amb99,Sho00,EMY03,EM04,Sho04,Kaw11,Kaw14,Kaw15,MN18,Nak16a,Nak16b,Kaw18,Liu18,HLQ20,CH20,HL20,NS20}.

Very recently, there has been studies on the mld from the perspective of K-stability theory. In particular, in \cite{HLS19,HLQ20}, \emph{normalized volumes} \cite{Li18} and \emph{Koll\'ar components} (in some references, called \emph{reduced components}) have played essential roles to prove some important cases of the ACC conjecture for mlds. Since the structure of the Koll\'ar components are very well-studied \cite{Sho96,Pro00,Kud01,Xu14,LX20}, we may propose the following natural folklore question:

\begin{ques}\label{ques: div compute mld is kc}
Let $(X\ni x,B)$ be an lc germ of dimension $\geq 2$ such that $X\ni x$ is klt. Under what conditions will there exist a divisor $E$ over $X\ni x$  such that $a(E,X,B)=\mld(X\ni x,B)$ and $E$ is a Koll\'ar component of $X\ni x$?
\end{ques}

In the paper, we show that Question \ref{ques: div compute mld is kc} always has a positive answer in dimension $2$:

\begin{thm}\label{thm: dim 2 compute mld is kc}
Let $(X\ni x,B)$ be an lc surface germ such that $X\ni x$ is klt. Then there exists a divisor $E$ over $X\ni x$ such that $a(E,X,B)=\mld(X\ni x,B)$ and $E$ is a Koll\'ar component of $X\ni x$.
\end{thm}

Regrettably, Question \ref{ques: div compute mld is kc} does not always have a positive answer in dimension $\geq 3$ even when $B=0$ due to Example \ref{ex: threefold mld not kc}.

For smooth surfaces, a modified version of Question \ref{ques: div compute mld is kc} was proved by Blum \cite[Theorem 1.2]{Blu16} and Kawakita \cite[Remark 3]{Kaw17}, who show that any divisor computing the mld is a potential lc place (see Definition \ref{defn: potential lc place} below) of the ambient variety, while Kawakita additionally shows that any such divisor is achieved by a weighted blow-up \cite[Theorem 1]{Kaw17}. With this in mind, we may ask the following folklore  question:

\begin{ques}\label{ques: div compute mld are lcp}
Let $(X\ni x,B)$ be an lc germ. Under what conditions will every divisor that computes the mld be a potential lc place?
\end{ques}

In this paper, we also answer Question \ref{ques: div compute mld are lcp} for surfaces:

\begin{thm}\label{thm: when all divisors computing mld are potential lc places}
Let $(X\ni x,B)$ be an lc surface germ. Then every divisor $E$ over $X\ni x$ such that $a(E,X,B)=\mld(X\ni x,B)$ is a potential lc place of $X\ni x$ if and only if $(X\ni x,B)$ is \textbf{not} of the following types:
\begin{enumerate}
    \item $B=0$ and $X\ni x$ is a $D_m$-type Du Val singularity for some integer $m\geq 5$, or
    \item $B=0$ and $X\ni x$ is an $E_m$-type Du Val singularity for some integer $m\in\{6,7,8\}$.
\end{enumerate}
\end{thm}

We say a few words about the intuition of Questions \ref{ques: div compute mld is kc} and \ref{ques: div compute mld are lcp}. Roughly speaking, a Koll\'ar component always admits a log Fano structure that is compatible with the local singularity, and a potential lc place always admits a log Calabi-Yau structure that is compatible with the local singularity. These structures allow us to use results in global birational geometry to study the behavior of the divisor and the local geometry of the singularity. On the other hand, we usually do not know whether a divisor calculating the mld is admits those good structures or not, and therefore, many powerful tools in global geometry are difficult to be applied to the study on the mlds of a singularity.

Therefore, getting a satisfactory answer for either Question \ref{ques: div compute mld is kc} or Question \ref{ques: div compute mld are lcp} could provide us possibilities to apply global geometry results to tackle the ACC conjecture or the LSC conjecture for mlds. In particular, since Koll\'ar components are well-studied in K-stability theory, with a satisfactory answer for Question \ref{ques: div compute mld is kc}, there are strong potentials for K-stability theory results to be applied to the study on mlds.

\medskip

Theorem \ref{thm: dim 2 compute mld is kc} and Theorem \ref{thm: when all divisors computing mld are potential lc places} follow from the following classification result on divisors computing mlds on lc surfaces:
\begin{thm}\label{thm: classification divisors computing mld on surfaces}
Let $(X\ni x,B)$ be an lc surface germ and $\mathcal{C}$ the set of all prime divisors over $X\ni x$ which compute $\mld(X\ni x,B)$. 
\begin{enumerate}
\item If $(X\ni x,B)$ is dlt, then:
\begin{enumerate}
    \item If $X\ni x$ is smooth or an $A$-type singularity, then any element of $\mathcal{C}$ is a Koll\'ar component of $X\ni x$.
    \item If $X\ni x$ is a $D_m$-type singularity for some integer $m\geq 4$ or an $E_m$-type singularity for some integer $m\in\{6,7,8\}$, let $f: Y\rightarrow X$ be the minimal resolution of $X\ni x$ and $\Dd(f)$ the dual graph of $f$. Then:
    \begin{enumerate}
        \item There exists a unique element $E\in\mathcal{C}$ that is  a Koll\'ar component of $X\ni x$, and $E$ is the unique fork of $\Dd(f)$.
        \item $\mathcal{C}\subset\Exc(f)$.
        \item If $B\not=0$ or $X\ni x$ is not Du Val, then:
        \begin{enumerate}
            \item Any element of $\mathcal{C}$ is a potential lc place of $X\ni x$.
            \item If $X\ni x$ is an $E_m$-type singularity, then $\mathcal{C}$ only contains the unique fork of $\Dd(f)$.
        \end{enumerate}
        \item If $B=0$ and $X\ni x$ is Du Val, then:
        \begin{enumerate}
            \item $\mathcal{C}=\Exc(f)$.
            \item If $X\ni x$ is a $D_m$-type singularity, then an element $F\in\mathcal{C}$ is a potential lc place of $X\ni x$ if and only if either $F$ is the fork of $\Dd(f)$, or the two branches which do contain $F$ both have length 1.
            \item If $X\ni x$ is an $E_m$-type singularity, then an element $F\in\mathcal{C}$ is a potential lc place of $X\ni x$ if and only if $F$ is the fork of $\Dd(f)$.
        \end{enumerate}
    \end{enumerate}
\end{enumerate}
\item If $(X\ni x,B)$ is not dlt but $X\ni x$ is klt, then:
\begin{enumerate}
    \item Any element of $\mathcal{C}$ is a potential lc place of $X\ni x$.
    \item There exists an element of $\mathcal{C}$ that is a Koll\'ar component of $X\ni x$.
    \item If $X\ni x$ is smooth, then any element of $\mathcal{C}$ is a Koll\'ar component of $\mld(X\ni x,B)$
\end{enumerate}
\item If $X\ni x$ is not klt, then any element of $\mathcal{C}$ is a potential lc place of $X\ni x$.
\end{enumerate}
\end{thm}

We hope that our results could inspire people to tackle Questions \ref{ques: div compute mld is kc} and \ref{ques: div compute mld are lcp}.

\begin{rem}
Some complementary examples of our main theorems are given in Section 6.
\end{rem}

\begin{rem}
Although the study on minimal log discrepancies was traditionally considered over $\mathbb C$, recently there has been some studies on the structure of minimal log discrepancies over fields of arbitrary characteristics (cf. \cite{Shi19,Che20,Ish20}). In this paper, the results hold over fields of arbitrary characteristics.
\end{rem}

\noindent\textbf{Acknowledgement}.  The authors would like to thank Christopher D. Hacon for useful discussions and encouragements. The first author would like to thank Harold Blum, Jingjun Han, Yuchen Liu, Chenyang Xu, and Ziquan Zhuang for inspiration to Questions \ref{ques: div compute mld is kc} and \ref{ques: div compute mld are lcp} and useful discussions. The authors would like to thank Jingjun Han for useful comments. Special thank to Ziquan Zhuang who kindly share Example \ref{ex: threefold mld not kc} after the first version of the paper was posted on arXiv. The authors were partially supported by NSF research grants no: DMS-1801851, DMS-1952522 and by a grant from the Simons Foundation; Award Number: 256202.

\section{Preliminaries}

We adopt the standard notation and definitions in \cite{KM98}.
\begin{defn}\label{defn: positivity}
	A \emph{pair} $(X,B)$ consists of a normal quasi-projective variety $X$ and an $\Rr$-divisor $B\ge0$ such that $K_X+B$ is $\Rr$-Cartier. If $B\in [0,1]$, then $B$ is called a \emph{boundary}.
	
	Let $E$ be a prime divisor on $X$ and $D$ an $\mathbb R$-divisor on $X$. We define $\mult_{E}D$ to be the \emph{multiplicity} of $E$ along $D$. Let $\phi:W\to X$
	be any log resolution of $(X,B)$ and let
	$$K_W+B_W:=\phi^{*}(K_X+B).$$
	The \emph{log discrepancy} of a prime divisor $D$ on $W$ with respect to $(X,B)$ is $1-\mult_{D}B_W$ and it is denoted by $a(D,X,B).$
	For any non-negative real number $\epsilon$, we say that $(X,B)$ is \emph{lc} (resp. \emph{klt}) if $a(D,X,B)\ge0$ (resp. $>0$) for every log resolution $\phi:W\to X$ as above and every prime divisor $D$ on $W$. We say that $(X,B)$ is \emph{dlt} if $a(D,X,B)>0$ for some log resolution $\phi:W\to X$ as above and every prime divisor $D$ on $W$. We say that $(X,B)$ is \emph{plt} if $a(D,X,B)>0$ for any exceptional prime divisor $D$ over $X$. 
	
	A \emph{germ} $(X\ni x,B)$ consists of a pair $(X,B)$ and a closed point $x\in X$. If $B=0$, the germ $(X\ni x,B)$ is usually represented by $X\ni x$. We say that $(X\ni x,B)$ is \emph{lc} (resp. \emph{klt}, \emph{dlt}, \emph{plt}) if $(X,B)$ is \emph{lc} (resp. \emph{klt}, \emph{dlt}, \emph{plt}) near $x$. We say that $(X\ni x,B)$ is \emph{smooth} if $X$ is smooth near $x$.  We say that $(X\ni x,B)$ is \emph{log smooth} if $X$ is log smooth near $x$.  
\end{defn}

\begin{defn}\label{defn: mld and alct} The \emph{minimal log discrepancy (mld)} of an lc germ $(X\ni x,B)$ is
	$$\mld(X\ni x,B):=\min\{a(E,X,B)\mid E \text{ is a prime divisor over } X\ni x\}.$$
\end{defn}

\begin{defn}[Plt blow-ups]\label{defn: reduced component}
	Let $(X\ni x,B)$ be a klt germ. A \emph{plt blow-up} of $(X\ni x,B)$ is a blow-up $f: Y\rightarrow X$ with the exceptional divisor $E$ over $X\ni x$, such that $(Y,f^{-1}_*B+E)$ is plt near $E$, and $-E$ is ample over $X$. The divisor $E$ is called a \emph{Koll\'ar component} (in some references, \emph{reduced component}) of $(X\ni x,B)$.
\end{defn}

\begin{defn}[Potential lc place]\label{defn: potential lc place}
Let $(X\ni x,B)$ be an lc germ. A \emph{potential lc place} of $(X\ni x,B)$ is a divisor $E$ over $X\ni x$, such that there exists $G\geq 0$ on $X$ such that $(X\ni x,B+G)$ is lc and $a(E,X,B+G)=0$.
\end{defn}

\begin{defn}
    A \emph{surface} is a normal quasi-projective variety of dimension $2$. A surface germ $X\ni x$ is called \emph{Du Val} if $\mld(X\ni x,0)=1$. 
    
Let $X\ni x$ be a klt surface germ of type $A$ (resp. $D,E$) and $m\geq 1$ (resp. $m\geq 4,m\in\{6,7,8\}$) an integer. Let $f$ be the minimal resolution of $X\ni x$. We say that $X\ni x$ is an $A_m$ (resp. $D_m,E_m$)-\emph{type singularity} if $\Exc(f)$ contains exactly $m$ prime divisors.
\end{defn}

For surfaces, to check that an extraction is a plt blow-up, we only need to control the singularity as the anti-ample requirement is automatically satisfied. The following lemma is well-known and we will use it many times:

\begin{lem}\label{lem: kc is plc place}
Let $(X\ni x,B)$ be a klt surface germ, $f: Y\rightarrow X$ an extraction of a prime divisor $E$, and $B_Y:=f^{-1}_*B$. Then:
\begin{enumerate}
    \item If $(Y,B_Y+E)$ is plt near $E$, then $E$ is a Koll\'ar component of $(X\ni x,B)$.
    \item If $(Y,B_Y+E)$ is lc near $E$, then $E$ is a potential lc place of $(X\ni x,B)$.
\end{enumerate} 
In particular, any Koll\'ar component of $(X\ni x,B)$ is a potential lc place of $(X\ni x,B)$.
\end{lem}
\begin{proof}
Since $(X\ni x,B)$ is a klt surface germ, $X$ is $\Qq$-factorial, so there exists an $f$-exceptional divisor $F\geq 0$ such that $-F$ is ample over $X$ \cite[Lemma 3.6.2(3)]{BCHM10}. Since $f$ only extracts $E$, $-E$ is ample over $X$. This implies (1).

Since $-E$ is ample over $X$, $-(K_Y+B_Y+E)$ is ample over $X$. We may pick a general $G_Y\sim_{\mathbb R,X}-(K_Y+B_Y+E)$ such that $(Y,B_Y+E+G_Y)$ is lc near $E$ and $K_Y+B_Y+E+G_Y\sim_{\mathbb R,X}0$. Let $G:=f_*G_Y$, then $(X\ni x,B+G)$ is lc and $E$ is a potential lc place of $(X\ni x,B)$. 
\end{proof}

\begin{defn}[Dual graph]\label{defn: crt}
Let $n$ be a non-negative integer, and $C=\cup_{i=1}^nC_i$ a collection of irreducible curves on a smooth surface $U$. We define the \emph{dual graph} $\Dd(C)$ of $C$ as follows.
\begin{enumerate}
    \item The vertices $v_i=v_i(C_i)$ of $\Dd(C)$ correspond to the curves $C_i$.
    \item For $i\neq j$,the vertices $v_i$ and $v_j$ are connected by $C_i\cdot C_j$ edges.
\end{enumerate}
In addition,
\begin{enumerate}
 \setcounter{enumi}{2}
    \item if we label each $v_i$ by the integer $e_i:=-C_i^2$, then $\Dd(C)$ is called the \emph{weighted dual graph} of $C$.
\end{enumerate}

A \emph{fork} of a dual graph is a curve $C_i$ such that $C_i\cdot C_j\geq 1$ for exactly three different $j\not=i$. A \emph{tail} of a dual graph is a curve $C_i$ such that $C_i\cdot C_j\geq 1$ for at most one $j\not=i$.

For any birational morphism $f: Y\rightarrow X$ between surfaces, let $E=\cup_{i=1}^nE_i$ be the reduced exceptional divisor for some non-negative integer $n$. We define $\Dd(f):=\Dd(E)$.

When we have a dual graph, we sometimes label $C_i$ near the vertex $v_i$. We sometimes use black dots in the dual graph to emphasize the corresponding curves that are not exceptional.
\end{defn}

\begin{defn}
Let $\Dd$ be a dual graph. If $\Dd$ looks like the following
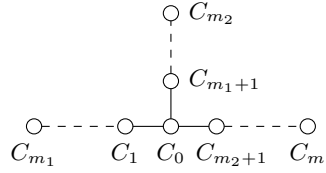
\begin{center}
    \begin{tikzpicture}
                                                    
         \draw (2.1,0) circle (0.1);
         \node [below] at (2.1,-0.1) {\footnotesize$C_{m_1}$};
         \draw[dashed] (2.2,0)--(3.2,0);
         \draw (3.3,0) circle (0.1);
         \node [below] at (3.3,-0.1) {\footnotesize$C_1$};
         \draw (3.4,0)--(3.8,0);

         \draw (3.9,+1.5) circle (0.1);
         \draw (3.9,0) circle (0.1);
         \draw (3.9,0.6) circle (0.1);
         \draw (4.5,0) circle (0.1);
         \node [below] at (4.7,-0.1) {\footnotesize$C_{m_2+1}$};
         \draw (5.7,0) circle (0.1);
         \node [below] at (5.7,-0.1) {\footnotesize$C_m$};

         \draw [dashed](3.9,+0.7)--(3.9,+1.4);
         \draw (3.9,0.1)--(3.9,0.5);
         \draw (4.0,0)--(4.4,0);
         \draw [dashed](4.6,0)--(5.6,0);

         \node [right] at (4,1.5) {\footnotesize$C_{m_2}$};
         \node [below] at (3.9,-0.1) {\footnotesize$C_0$};
         \node [right] at (4,0.6) {\footnotesize$C_{m_1+1}$};
    \end{tikzpicture}
\end{center}
for some integers $m>m_2>m_1$, then $\cup_{i=1}^{m_1}C_i$, $\cup_{i=m_1+1}^{m_2}C_i$, and $\cup_{i=m_2+1}^{m}C_i$ will be called the \emph{branches} of $\Dd$. The \emph{length} of a branch is the number of irreducible curves in this branch.
\end{defn}

We will use the following lemmas many times in this paper:
 
\begin{lem}\label{lem: intersection number surface}
Let $X\ni x$ be a smooth germ, $f: Y\rightarrow X$ the smooth blow-up of $X\ni x$ with exceptional divisor $E$, and $C$ and $D$ two $\Rr$-divisors on $X$ without common irreducible components. Then
\begin{align*}
    (C\cdot D)_x&=\sum_{y\in f^{-1}(x)}(f^{-1}_*C\cdot f^{-1}_*D)_y+(f^{-1}_*C\cdot E)(f^{-1}_*D\cdot E)\\
    &=\sum_{y\in f^{-1}(x)}(f^{-1}_*C\cdot f^{-1}_*D)_y+\mult_xC\cdot\mult_xD.
\end{align*}
\end{lem}
\begin{proof}
Possibly shrinking $X$ to a neighborhood of $x$ and shrinking $Y$ to a neighborhood over $x$, we may assume that $(C\cdot D)_x=C\cdot D$ and $\sum_{y\in f^{-1}(x)}(f^{-1}_*C\cdot f^{-1}_*D)_y=f^{-1}_*C\cdot f^{-1}_*D$. Thus
$$0=f^*C\cdot E=(f^{-1}_*C+(\mult_xC)E)\cdot E=f^{-1}_*C\cdot E-\mult_xC$$
and
$$0=f^*D\cdot E=(f^{-1}_*D+(\mult_xC)E)\cdot E=f^{-1}_*D\cdot E-\mult_xD.$$
By the projection formula, 
$$C\cdot D=f^*C\cdot f^{-1}_*D=f^{-1}_*C\cdot f^{-1}_*D+(\mult_xC)E\cdot f^{-1}_*D=f^{-1}_*C\cdot f^{-1}_*D+\mult_xC\cdot \mult_xD.$$
\end{proof}

\begin{lem}[{cf. \cite[Lemma 3.41, Corollary 4.2]{KM98}}]\label{lem: intersection number compare coefficients}
Let $U$ be a smooth surface and $C=\cup_{i=1}^mC_i$ a connected proper curve on $U$. Assume that the intersection matrix $\{(C_i\cdot C_j)\}_{1\leq i,j\leq m}$ is negative definite. Let $A=\sum_{i=1}^m a_iC_i$ and $H=\sum_{i=1}^m b_iC_i$ be $\Rr$-linear combinations of the curves $C_i$. Assume that $H\cdot C_i\leq A\cdot C_i$ for every $i$, then either $a_i=b_i$ for each $i$ or $a_i<b_i$ for each $i$.
\end{lem}

\section{Divisors computing mlds over smooth surface germs}

In this section, we study the behavior of divisors computing mlds over a smooth surface germ. The following Definition-Lemma greatly simplify the notation in the rest of the paper and we will use it many times.

\subsection{Definitions and lemmas}

\begin{deflem}\label{deflem: fE}
Let $(X\ni x,B)$ be a smooth lc surface germ and $E$ a divisor over $X\ni x$. By \cite[Lemma 2.45]{KM98}, there exists a unique positive integer $n=n(E)$ and a unique sequence of smooth blow-ups
$$X_E:=X_{n,E}\xrightarrow{f_{n,E}}X_{n-1,E}\xrightarrow{f_{n-1,E}}\dots\xrightarrow{f_{1,E}}X_{0,E}:=X$$
such that $E$ is on $X_E$ and each $f_{i,E}$ is the smooth blow-up at $\Center_{X_{i-1,E}}E$. We define $f_E:=f_{1,E}\circ f_{2,E}\dots\circ f_{n,E}$, $E^i$ the exceptional divisor of $f_{i,E}$ for each $i$, and $E^i_j$ the strict transform of $E^i$ on $X_{j,E}$ for any $i\leq j$. In particular, $E^i_i=E^i$ for each $i$ and $E=E^n=E^n_n$.
\end{deflem}

\begin{rem}\label{rem: elementary lemma dual graph smooth blow up}
We need the following facts many times, which are elementary and we omit the proof. Let $(X\ni x,B)$ be a smooth lc surface germ, $E$ a divisor over $X\ni x$, and $n:=n(E)$. Then for any $i\leq j$ such that $i,j\in\{1,2,\dots,n\}$,
\begin{enumerate}
\item $E^i_j$ is a smooth rational curve,
    \item $\cup_{k=1}^iE^k_i$ is simple normal crossing,
    \item $(E^i)^2=-1$, and
    \item $(E^i_j)^2\leq -2$ when $i<j$.
\end{enumerate}
\end{rem}

\begin{lem}\label{lem: surface blow-up a divisor is 2-multiplicity}
Let $(X\ni x,B)$ be a smooth surface germ and $f: Y\rightarrow X$ the smooth blow-up at $x$ with exceptional divisor $E$. Then $a(E,X,B)=2-\mult_xB$. In particular, $\mld(X\ni x,B)\leq 2-\mult_xB$.
\end{lem}
\begin{proof}
This immediately follows from \cite[Lemma 2.29]{KM98}.
\end{proof}

\begin{lem}\label{lem: mult is leq intersection number}
Let $X\ni x$ be a smooth surface germ, $B\geq 0$ an $\Rr$-divisor on $X$ and $C$ a prime divisor on $X$. Then $(B\cdot C)_x\geq\mult_xB\cdot\mult_xC$.
\end{lem}
\begin{proof}
It immediately follows from \cite[Excercise 5.4(a)]{Har77}
\end{proof}

\begin{lem}\label{lem: interseciton<1 imply mld does not need blow up}
Let $a\in [0,1]$ be a real number and $(X\ni x,\Delta:=B+aC)$ a smooth surface germ, where $B\geq 0$ is an $\Rr$-divisor and $C$ is a prime divisor such that $C\not\subset\Supp B$ and $C$ is smooth at $x$. Assume that $(B\cdot C)_{x}<1$, then $\mld(X\ni x,\Delta)>1-a$.
\end{lem}
\begin{proof}
We only need to show that for any positive integer $n$ and any sequence of smooth blow-ups
$$X_n\xrightarrow{f_n}X_{n-1}\xrightarrow{f_{n-1}}\dots\xrightarrow{f_1}X_0:=X$$
over $X\ni x$ with exceptional divisors $E_k:=\Exc(f_k)$ for each $k$, we have $a(E_k,X,\Delta)>1-a$.

In the following, we show that $a(E_k,X,\Delta)>1-a$ for each $k$ by applying induction on $n$. When $n=1$, by Lemmas \ref{lem: surface blow-up a divisor is 2-multiplicity} and \ref{lem: mult is leq intersection number}, 
$$a(E_1,X,\Delta)=2-\mult_x(B+aC)=2-a-\mult_xB\geq 2-a-(B\cdot C)_x>1-a.$$
Therefore, we may assume that $n\geq 2$, and when we blow-up at most $n-1$ times, each divisor we have extracted has log discrepancy $>1-a$. In particular, $a(E_k,X,\Delta)>1-a$ for any $k\in\{1,2,\dots,n-1\}$.

Let $x_1:=\Center_XE_n$, and $B_1,C_1,\Delta_1$ the strict transforms of $B,C$ and $\Delta$ on $X_1$ respectively. Then $x_1\in E_1$. There are three cases:

\medskip

\noindent\textbf{Case 1}. $a+\mult_xB-1<0$. By Lemma \ref{lem: intersection number surface},
$$(B_1\cdot C_1)_{x_1}\leq (B\cdot C)_x-B_1\cdot E_1=(B\cdot C)_x-\mult_xB<1-\mult_xB\leq 1.$$
By induction hypothesis for the germ $(X_1\ni x_1,\Delta_1)$ and blowing-up at most $n-1$ times, we have
$$a(E_n,X,\Delta)=a(E_n,X_1,\Delta_1+(a+\mult_xB-1)E_1)\geq a(E_n,X_1,\Delta_1)>1-a$$
and finish the proof for \textbf{Case 1}.

\medskip

\noindent\textbf{Case 2}. $a+\mult_xB-1\geq 0$ and $x_1\in E_1\cap C_1$. Let $\tilde B_1:=B_1+(a+\mult_xB-1)E_1$. Then 
$$K_{X_1}+\tilde B_1+aC_1=f_1^*(K_X+\Delta).$$
We have
$$(\tilde B_1\cdot C_1)_{x_1}=(B_1\cdot C_1)_{x_1}+(a+\mult_xB-1)(E_1\cdot C_1)_{x_1}.$$
Since $C$ is smooth at $x$, $E_1\cup C_1$ is snc at $x_1$, so $(E_1\cdot C_1)_{x_1}=1$. By Lemma \ref{lem: intersection number surface},
$$(B_1\cdot C_1)_{x_1}\leq (B\cdot C)_x-B_1\cdot E_1=(B\cdot C)_x-\mult_xB<1-\mult_xB.$$
Thus 
$$(\tilde B_1\cdot C_1)_{x_1}<1-\mult_xB+(a+\mult_xB-1)=a\leq 1.$$
By induction hypothesis for the germ $(X_1\ni x_1,\tilde B_1+aC_1)$ and blowing-up at most $n-1$ times,
$$a(E_n,X,\Delta)=a(E_n,X_1,\tilde B_1+aC_1)>1-a,$$
and we finish the proof for \textbf{Case 2}.

\medskip

\noindent\textbf{Case 3}. $a+\mult_xB-1\geq 0$, $x_1\in E_1$, but $x_1\not\in C_1$. By Lemma \ref{lem: intersection number surface},
$$(B_1\cdot E_1)_{x_1}\leq B_1\cdot E_1=\mult_xB\leq (B\cdot C)_x<1$$
By induction hypothesis for the germ $(X_1\ni x_1,B_1+(a+\mult_xB-1)E_1)$ and blowing-up at most $n-1$ times and apply Lemma \ref{lem: mult is leq intersection number}, 
$$a(E_n,X,\Delta)=a(E_n,X_1,B_1+(a+\mult_xB-1)E_1)\geq 1-(a+\mult_xB-1)>1-a,$$
and we finish the proof for \textbf{Case 3}.
\end{proof}

\begin{lem}\label{lem: two times intersction number}
Let $X\ni x$ be a smooth germ, $B\geq 0$ an $\Rr$-divisor on $X$ and $C$ a prime divisor on $X$ such that $C\not\subset\Supp B$ and $C$ is smooth at $x$. Let $g_1: X_1\rightarrow X$ be the smooth blow-up of $x$ with exceptional divisor $E_1$, $C_1:=(g_1^{-1})_*C$, and $B_1:=(g_1^{-1})_*B$. Let $g_2: X_2\rightarrow X_1$ be the smooth blow-up at $x_1:=C_1\cap E_1$ with exceptional divisor $E_2$ and $B_2:=(g_2^{-1})_*B_1$. Then
$$(B\cdot C)_x\geq 2(B_2\cdot E_2).$$
\end{lem}
\begin{proof}
Let $C_2:=(g_2^{-1})_*C_1$ and $E_1':=(g_2^{-1})_*E_1$. By Lemma \ref{lem: intersection number surface},
$$B_1\cdot E_1=B_2\cdot E'_1+(E'_1\cdot E_2)(B_2\cdot E_2)\geq B_2\cdot E_2$$
and
$$(B_1\cdot C_1)_{x_1}=\sum_{y\in g_2^{-1}(x_1)}(B_2\cdot C_2)_y+(B_2\cdot E_2)(C_2\cdot E_2)\geq (B_2\cdot E_2)(C_2\cdot E_2)= B_2\cdot E_2.$$
Thus
$$(B\cdot C)_x=\sum_{y\in g_1^{-1}(x)}(B_1\cdot C_1)_{y}+B_1\cdot E_1\geq (B_1\cdot C_1)_{x_1}+B_1\cdot E_1\geq 2(B_2\cdot E_2).$$
\end{proof}

\subsection{Dual graph of \texorpdfstring{$f_E$}{}}

Roughly speaking, this subsection shows that the dual graph of $f_E$ is almost always a chain when $E$ is a divisor which computes the mld. We remark that \cite[Lemma 3.18]{HL20} shows that there exists such an $E$ such that the dual graph of $f_E$ is a chain, while we show that for any such $E$, the dual graph of $f_E$ is a chain.

We need the following result of Kawakita. Notice that the ``in particular" part of the following theorem is immediate from the construction in \cite[Remark 3]{Kaw17}.
\begin{thm}[{\cite[Theorem 1, Remark 3]{Kaw17}}]\label{thm: Kaw17 1.1}
Let $(X\ni x,B)$ be a smooth lc surface germ and $E$ a divisor over $X\ni x$ such that $a(E,X,B)=\mld(X\ni x,B)$. Then there exists a weighted blow-up of $X\ni x$ which extracts $E$. In particular, $E$ is a Koll\'ar component of $X\ni x$.
\end{thm}

\begin{lem}\label{lem: exc is a chain}
Let $(X\ni x,B)$ be a smooth lc surface germ and $E$ a divisor over $X\ni x$ such that $a(E,X,B)=\mld(X\ni x,B)$. Then $\Dd(f_E)$ is a chain.
\end{lem}
\begin{proof}
By Theorem \ref{thm: Kaw17 1.1}, there exists a weighted blow-up $f: Y\rightarrow X$ which extracts $E$. $E$ contains at most two singular points of $Y$ which are cyclic quotient singularities, and locally analytically, $E$ is one coordinate line of each cyclic quotient singularity. Let  $g: W\rightarrow Y$ be the minimal resolution of $Y$ near $E$, then $f\circ g=f_E$ and $\Dd(f_E)$ is a chain.
\end{proof}

\begin{lem}\label{lem: blow-up not intersection is chain}
Let $a\in [0,1]$ be a real number. Assume that
\begin{enumerate}
\item $(X\ni x,\Delta:=B+aC)$ a smooth lc surface germ, where $B\geq 0$ is an $\Rr$-divisor and $C$ is a prime divisor,
\item  $C\not\subset\Supp B$ and $C$ is smooth at $x$, and
\item $(B\cdot C)_x<2$.
\end{enumerate}
Then any divisor $E$ over $X\ni x$ such that $a(E,X,\Delta)=\mld(X\ni x,\Delta)$ satisfies the following. Let $C_E:=(f_E^{-1})_*C$, then $C_E\cup\Exc(f_E)$ is a chain and $C_E$ is one tail of $C_E\cup\Exc(f_E)$. In particular, the dual graph of $C_E\cup\Exc(f_E)$ is the following:
\begin{center}
    \begin{tikzpicture}
                                                    
         \draw[fill=black] (1.5,0) circle (0.1);
         \draw (1.6,0)--(2,0);
         \draw (2.1,0) circle (0.1);
         \draw[dashed] (2.2,0)--(3.2,0);
         \draw (3.3,0) circle (0.1);
         \draw (3.4,0)--(3.8,0);

         \draw (3.9,0) circle (0.1);
         
         
    \end{tikzpicture}
\end{center}
Here $C_E$ is denoted by the black circle.
\end{lem}

\begin{proof}
By the construction of $f_E$, if  $C_E\cup\Exc(f_E)$ is a chain, then $C_E$ is a tail of  $C_E\cup\Exc(f_E)$. Let $n:=n(E)$ and let $C_i$ be the strict transform of $C$ on $X_{i,E}$ for each $i\in\{0,1,\dots,n\}$. Since $C$ is smooth at $x$, by the construction of $f_E$, $C_i\cup_{k=1}^iE^k_i$ is simple normal crossing over a neighborhood of $x$ and its dual graph does not contain a circle for each  $i\in\{0,1,\dots,n\}$. Since $a(E,X,\Delta)=\mld(X\ni x,\Delta)$, $a(E^i,X,\Delta)\geq a(E,X,\Delta)$ for every $i\in\{1,2,\dots,n\}$. 

Suppose that the lemma does not hold, then there exists a positive integer $m\in\{1,2,\dots,n\}$, such that $C_i\cup_{k=1}^iE^k_i$ is a chain, $C_i$ is a tail of $C_i\cup_{k=1}^iE^k_i$  for any $i\in\{0,1,\dots,m-1\}$, and $C_m\cup_{k=1}^mE^k_m$ is not a chain. Since any graph without a circle that is not a chain contains at least $4$ vertices, $m\geq 3$. 

By Lemma \ref{lem: exc is a chain}, $\cup_{k=1}^nE^k_n$ is a chain, hence $\cup_{k=1}^mE^k_m$ is a chain. Since $C_m\cup_{k=1}^mE^k_m$ is not a chain,
$x_{m-1}:=\Center_{X_{m-1,E}}E^m\in E^{m-1}\backslash C_{m-1}$, and for any integer $i\in\{1,2,\dots,m-1\}$, $f_{i,E}$ is the smooth blow-up of a point $x_{i-1}\in C_{i-1}$, and if $i\in\{2,3,\dots,m-1\}$, then $f_{i,E}$ is the smooth blow-up of $C_{i-1}\cap E^{i-1}$. In particular, since $m\geq 3$, $x_{m-3}\in C_{m-3}$ and $x_{m-2}\in C_{m-2}\cap E^{m-2}$.

\begin{claim}\label{claim: local int geq 1 on exc divisors}
$(B_{m-1}\cdot E^{m-1})_{x_{m-1}}\geq 1$.
\end{claim}
\begin{proof}
Let $a_{m-1}:=\max\{0,1-a(E^{m-1},X,\Delta)\}$. Since $(X\ni x,\Delta)$ is lc, $a_{m-1}\in [0,1]$. If $(B_{m-1}\cdot E^{m-1})_{x_{m-1}}<1$, then by Lemma \ref{lem: interseciton<1 imply mld does not need blow up},
\begin{align*}
    \mld(X\ni x,\Delta)&=a(E^n,X,\Delta)=a(E^n,X^{m-1},B_{m-1}+(1-a(E^{m-1},X,\Delta))E^{m-1})\\
    &\geq a(E^n,X^{m-1},B_{m-1}+a_{m-1}E^{m-1})\\
    &\geq\mld(X_{m-1}\ni x_{m-1},B_{m-1}+a_{m-1}E^{m-1})\\
    &>1-a_{m-1}=a(E^{m-1},X,\Delta),
\end{align*}
a contradiction.
\end{proof}

\noindent\textit{Proof of Lemma \ref{lem: blow-up not intersection is chain} continued}. Since $m\geq 3$, by Lemmas \ref{lem: intersection number surface}, \ref{lem: two times intersction number}, and Claim \ref{claim: local int geq 1 on exc divisors}, 
$$(B\cdot C)_x\geq (B_{m-3}\cdot C_{m-3})_{x_{m-3}}\geq 2(B_{m-1}\cdot E^{m-1})\geq 2(B_{m-1}\cdot E^{m-1})_{x_{m-1}}\geq 2,$$
which contradicts our assumptions.
\end{proof}

\begin{lem}\label{lem: blow-up not intersection is chain two tails}
Let $l,r\in [0,1]$ be two real numbers. Assume that
\begin{enumerate}
\item $(X\ni x,\Delta:=B+lL+rR)$ a smooth lc surface germ, where $B\geq 0$ is an $\Rr$-divisor and $L,R$ are two different prime divisors.
\item  $L\not\subset\Supp B, R\not\subset\Supp B$, and $(X\ni x,L+R)$ is log smooth, and
\item either $(B\cdot L)_x<1$ or $(B\cdot R)_x<1$.
\end{enumerate}
Then any divisor $E$ over $X\ni x$ such that $a(E,X,\Delta)=\mld(X\ni x,\Delta)$ satisfies the following. Let $L_E:=(f_E^{-1})_*L$ and $R_E:=(f_E^{-1})_*R$, then $L_E\cup R_E\cup\Exc(f_E)$ is a chain and $L_E,R_E$ are the tails of $L_E\cup R_E\cup\Exc(f_E)$. In particular, the dual graph of $L_E\cup R_E\cup\Exc(f_E)$ is the following:
\begin{center}
    \begin{tikzpicture}
                                                    
         \draw[fill=black] (1.5,0) circle (0.1);
         \draw (1.6,0)--(2,0);
         \draw (2.1,0) circle (0.1);
         \draw[dashed] (2.2,0)--(3.2,0);
         \draw (3.3,0) circle (0.1);
         \draw (3.4,0)--(3.8,0);

         \draw[fill=black] (3.9,0) circle (0.1);
         
         
    \end{tikzpicture}
\end{center}
Here $L_E$ and $R_E$ are denoted by the left black circle and the right black circle respectively.
\end{lem}
\begin{proof}
By the construction of $f_E$, if $L_E\cup R_E\cup\Exc(f_E)$ is a chain, then $L_E,R_E$ are the tails of $L_E\cup R_E\cup\Exc(f_E)$. Let $n:=n(E)$ and let $L_i,R_i$ be the strict transforms of $L_i,R_i$ on $X_{i,E}$ for each $i\in\{0,1,\dots,n\}$. Since $(X\ni x,L+R)$ is log smooth and by the construction of $f_E$, $L_i\cup R_i\cup_{k=1}^iE^k_i$ is simple normal crossing over a neighborhood of $x$ and its dual graph does not contain a circle for each $i\in\{0,1,\dots,n\}$. Since $a(E,X,\Delta)=\mld(X\ni x,\Delta)$, $a(E^i,X,\Delta)\geq a(E,X,\Delta)$ for every $i\in\{1,2,\dots,n\}$.

Suppose that the lemma does not hold, then there exists a positive integer $m\in\{1,2,\dots,n\}$, such that $L_i\cup R_i\cup_{k=1}^iE^k_i$ is a chain and $L_i,R_i$ are the tails of $L_i\cup R_i\cup_{k=1}^iE^k_i$ for any $i\in\{0,1,\dots,m-1\}$, and $L_m\cup R_m\cup_{k=1}^mE^k_m$ is not a chain. Since any graph without a circle that is not a chain contains at least $4$ points, $m\geq 2$. 

By Theorem \ref{thm: Kaw17 1.1}, $\cup_{k=1}^nE^k_n$ is a chain, hence $\cup_{k=1}^mE^k_m$ is a chain. By the construction of $f_E$ and the choice of $m$, $x_{m-1}:=\Center_{X_{m-1,E}}E^m\in E^{m-1}\backslash (L_{m-1}\cup R_{m-1})$.

\begin{claim}\label{claim: local int geq 1 on exc divisors second claim}
$(B_{m-1}\cdot E^{m-1})_{x_{m-1}}\geq 1$.
\end{claim}
\begin{proof}
Let $a_{m-1}:=\max\{0,1-a(E^{m-1},X,\Delta)\}$. Since $(X\ni x,\Delta)$ is lc, $a_{m-1}\in [0,1]$. If $(B_{m-1}\cdot E^{m-1})_{x_{m-1}}<1$, then by Lemma \ref{lem: interseciton<1 imply mld does not need blow up},
\begin{align*}
    \mld(X\ni x,\Delta)&=a(E^n,X,\Delta)=a(E^n,X^{m-1},B_{m-1}+(1-a(E^{m-1},X,\Delta))E^{m-1})\\
    &\geq a(E^n,X^{m-1},B_{m-1}+a_{m-1}E^{m-1})\\
    &\geq\mld(X_{m-1}\ni x_{m-1},B_{m-1}+a_{m-1}E^{m-1})\\
    &>1-a_{m-1}=a(E^{m-1},X,\Delta),
\end{align*}
a contradiction.
\end{proof}

\noindent\textit{Proof of Lemma \ref{lem: blow-up not intersection is chain two tails} continued}.
Since $m\geq 2$, by Lemma \ref{lem: intersection number surface} and Claim \ref{claim: local int geq 1 on exc divisors second claim}, 
$$(B\cdot L)_x\geq (B_{m-2}\cdot L_{m-2})_{x_{m-2}}\geq (B_{m-1}\cdot E^{m-1})_{x_{m-1}}\geq 1$$
and
$$(B\cdot R)_x\geq (B_{m-2}\cdot R_{m-2})_{x_{m-2}}\geq (B_{m-1}\cdot E^{m-1})_{x_{m-1}}\geq 1$$
which contradict our assumptions.
\end{proof}

\section{Classification of divisors computing mlds}

\subsection{A key lemma}

The following lemma is similar to \cite[Theorem 4.15]{KM98} and plays an important role in the proof of our main theorems.
\begin{lem}\label{lem: km98 4.15 generalization}
Let $m$ be a non-negative integer, $(X\ni x,B)$ a plt surface germ, $f: Y\rightarrow X$ the minimal resolution of $X\ni x$, and $B_Y:=f^{-1}_*B$. Then
\begin{enumerate}
\item for any prime divisor $F\subset\Exc(f)$, $B_Y\cdot F<2$,
\item there exists at most one prime divisor $F\subset\Exc(f)$ such that $B_Y\cdot F\geq 1$, and
\item if $E\subset\Exc(f)$ is a prime divisor such that  $B_Y\cdot E\geq 1$, then $X\ni x$ is an $A$-type singularity and $E$ is a tail of $\Dd(f)$.
\end{enumerate}
\end{lem}

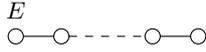
\begin{center}
    \begin{tikzpicture}
                                                    
         \draw (1.5,0) circle (0.1);
         \node [above] at (1.5,0.1) {\footnotesize$E$};
         \draw (1.6,0)--(2,0);
         \draw (2.1,0) circle (0.1);
         \draw[dashed] (2.2,0)--(3.2,0);
         \draw (3.3,0) circle (0.1);
         \draw (3.4,0)--(3.8,0);

         \draw (3.9,0) circle (0.1);
         
         
    \end{tikzpicture}
\end{center}

\begin{proof}
Let $F_1,\dots,F_m$ be the prime exceptional divisors of $h$ for some positive integer $m$, and let $v_1,\dots,v_m$ be the vertices corresponding to $F_1,\dots,F_m$ in $\Dd(f)$ respectively. We construct an extended graph $\bar\Dd(f)$ in the following way:
\begin{itemize}
    \item The vertices of $\bar\Dd(f)$ are $v_0,v_1,\dots,v_m$.
    \item For any $i,j\in\{1,2,\dots,m\}$, $v_i$ and $v_j$ are connected by a line if and only if $v_i$ and $v_j$ are connected by a line in $\Dd(f)$.
    \item For any $i\in\{1,2,\dots,m\}$, $v_0$ and $v_i$ are connected by $\lfloor B_Y\cdot F_i\rfloor$ lines.
\end{itemize}
Moreover, we may write
$$K_Y+B_Y-\sum_{i=1}^ma_iF_i=f^*(K_X+B),$$
where $a_i:=a(F_i,X,B)-1$. Since $(X\ni x,B)$ is plt and $f$ is the minimal resolution of $X\ni x$, $0\geq a_i>-1$ for each $i$. Let $A:=\sum_{i=1}^ma_iF_i$.

If $\bar\Dd(f)$ is not connected, then $B_Y\cdot F_i<1$ for each $i$ and there is nothing left to prove. Therefore, we may assume that $\bar\Dd(f)$ is connected.

\begin{claim}\label{claim: extended graph does not contain circle}
 $\bar\Dd(f)$ does not contain a circle.
\end{claim}
\begin{proof}
Suppose that $\bar\Dd(f)$ contains a circle. We let $v_{k_0}:=v_0,v_{k_1},\dots,v_{k_r}$ be the vertices of this circle for some positive $r$ such that $\bar\Dd(f)$ contains one of the following sub-graphs:
 
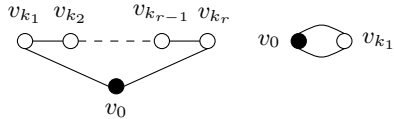
\begin{center}
    \begin{tikzpicture}
                                                    
         \draw (1.5,0) circle (0.1);
         \node [above] at (1.5,0.1) {\footnotesize$v_{k_1}$};
         \draw (1.6,0)--(2,0);
         \draw (2.1,0) circle (0.1);
         \node [above] at (2.1,0.1) {\footnotesize$v_{k_2}$};
         \draw[dashed] (2.2,0)--(3.2,0);
         \draw[fill=black] (2.7,-0.6) circle (0.1);
         \node [below] at (2.7,-0.7) {\footnotesize$v_0$};
         \draw (2.8,-0.6)--(3.9,-0.1);
         \draw (2.6,-0.6)--(1.5,-0.1);
        
         \draw (3.3,0) circle (0.1);
         \node [above] at (3.3,0.1) {\footnotesize$v_{k_{r-1}}$};
         \draw (3.4,0)--(3.8,0);

         \draw (3.9,0) circle (0.1);
         
         
         \node [above] at (4.0,0.1) {\footnotesize$v_{k_r}$};
        
         \draw (5.1,0)[fill=black] circle (0.1);
         \node [right] at (5.8,0) {\footnotesize$v_{k_1}$};
         \draw (5.7,0) circle (0.1);
         \node [left] at (5.0,0) {\footnotesize$v_0$};
         \draw (5.1,0.1)..controls (5.4,0.25)..(5.7,0.1);
         \draw (5.1,-0.1)..controls (5.4,-0.25)..(5.7,-0.1);
        
    \end{tikzpicture}
\end{center}
Let $H:=-\sum_{i=1}^rF_{k_i}$. Then
$$H\cdot F_{k_i}=-2-F_{k_i}^2=K_Y\cdot F_{k_i}\leq (K_Y+B_Y)\cdot F_{k_i}=A\cdot F_{k_i}$$
for each $i\in\{2,3,\dots,r-1\}$,
$$H\cdot F_{k_i}=-1-F_{k_i}^2=K_Y\cdot F_{k_i}+1\leq (K_Y+B_Y)\cdot F_{k_i}=A\cdot F_{k_i}$$
for each $i\in\{1,r\}$ when $r\geq 2$,
$$H\cdot F_{k_1}=-F_{k_1}^2=K_Y\cdot F_{k_i}+2\leq (K_Y+B_Y)\cdot F_{k_i}=A\cdot F_{k_i}$$
when $r=1$, and
$$H\cdot F_{i}\leq 0\leq(K_Y+B_Y)\cdot F_{i}=A\cdot F_{i}$$
for each $i\not\in\{k_1,\dots,k_r\}$. By Lemma \ref{lem: intersection number compare coefficients}, $a_{k_i}\leq -1$ for each $i$, a contradiction.
\end{proof}

\begin{claim}\label{claim: extended graph does not contain fork}
$\bar\Dd(f)$ does not contain a fork.
\end{claim}

\begin{proof}
Assume that $\bar\Dd(f)$ contains a fork. 
By Claim \ref{claim: extended graph does not contain circle},  $\bar\Dd(f)$ does not contain a circle. Therefore, $v_0$ is not a fork of $\bar\Dd(f)$. In particular, there exist a positive integer $r$ and vertices $v_{k_0}:=v_0,v_{k_1},\dots,v_{k_r},v_{l_1},v_{l_2}$ such that $\bar\Dd(f)$ contains the following sub-graph:
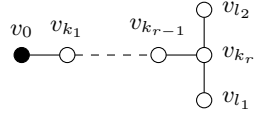
\begin{center}
    \begin{tikzpicture}
                                                    
         \draw[fill=black] (1.5,0) circle (0.1);
         \node [above] at (1.5,0.1) {\footnotesize$v_0$};
         \draw (1.6,0)--(2,0);
         \draw (2.1,0) circle (0.1);
         \node [above] at (2.1,0.1) {\footnotesize$v_{k_1}$};
         \draw[dashed] (2.2,0)--(3.2,0);
         \draw (3.3,0) circle (0.1);
         \node [above] at (3.3,0.1) {\footnotesize$v_{k_{r-1}}$};
         \draw (3.4,0)--(3.8,0);

         \draw (3.9,0) circle (0.1);
         \draw (3.9,0.6) circle (0.1);
         \draw (3.9,-0.6) circle (0.1);
         
         \draw (3.9,0.1)--(3.9,0.5);
         \draw (3.9,-0.1)--(3.9,-0.5);
         
         \node [right] at (4,0) {\footnotesize$v_{k_r}$};
         \node [right] at (4,0.6) {\footnotesize$v_{l_2}$};
         \node [right] at (4,-0.6) {\footnotesize$v_{l_1}$};
    \end{tikzpicture}
\end{center}
Let $H:=-\sum_{i=1}^rF_{k_i}-\frac{1}{2}F_{l_1}-\frac{1}{2}F_{l_2}$. Then
$$H\cdot F_{l_i}=-1-\frac{1}{2}F_{l_i}^2\leq -2-F^2_{l_i}=K_Y\cdot F_{l_i}\leq(K_Y+B_Y)\cdot F_{l_i}=A\cdot F_{l_i}$$
for each $i\in\{1,2\}$,
$$H\cdot F_{k_i}=-2-F_{k_i}^2=K_Y\cdot F_{k_i}\leq (K_Y+B_Y)\cdot F_{k_i}=A\cdot F_{k_i}$$
for each $i\in\{2,3,\dots,r-1\}$,
$$H\cdot F_{k_i}=-1-F_{k_i}^2=K_Y\cdot F_{k_i}+1\leq (K_Y+B_Y)\cdot F_{k_i}=A\cdot F_{k_i}$$
for $i=1$, and
$$H\cdot F_{i}\leq 0\leq(K_Y+B_Y)\cdot F_{i}=A\cdot F_{i}.$$
for each $i\not\in\{k_1,\dots,k_r,l_1,l_2\}$. By Lemma \ref{lem: intersection number compare coefficients}, $a_{k_r}\leq -1$, a contradiction.
\end{proof}

\noindent\textit{Proof of Lemma \ref{lem: km98 4.15 generalization} continued}. 
By Claims \ref{claim: extended graph does not contain circle} and \ref{claim: extended graph does not contain fork}, $\bar\Dd(f)$ does not contain a circle or a fork. Therefore, $\bar\Dd(f)$ is a chain. Since $\Dd(f)$ is connected and $\bar\Dd(f)$ has $\Dd(f)$ as a sub-graph, $\Dd(f)$ is a chain and $v_0$ is a tail of $\bar\Dd(f)$. The lemma immediately follows from the structure of $\bar\Dd(f)$ and $\Dd(f)$.
\end{proof}

\subsection{\texorpdfstring{$A$}{}-type singularities}

\begin{lem}\label{lem: a type is plt}
Let $X\ni x$ be a surface germ of $A$-type, $E$ a prime divisor on $X$, $f: Y\rightarrow X$ the minimal resolution of $X\ni x$, and $E_Y:=f^{-1}_*E$. Let $F_1,\dots,F_m$ be the prime exceptional divisors over $X\ni x$. Assume that $E_Y\cup_{i=1}^mF_i$ is simple normal crossing over a neighborhood of $x$ and the dual graph of $E_Y\cup_{i=1}^mF_i$ is the following:
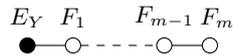
\begin{center}
    \begin{tikzpicture}
                                                    
         \draw[fill=black] (1.5,0) circle (0.1);
         \node [above] at (1.5,0.1) {\footnotesize$E_Y$};
         \draw (1.6,0)--(2,0);
         \draw (2.1,0) circle (0.1);
         \node [above] at (2.1,0.1) {\footnotesize$F_1$};
         \draw[dashed] (2.2,0)--(3.2,0);
         \draw (3.3,0) circle (0.1);
         \node [above] at (3.3,0.1) {\footnotesize$F_{m-1}$};
         \draw (3.4,0)--(3.8,0);

         \draw (3.9,0) circle (0.1);
         
         
         \node [above] at (4.0,0.1) {\footnotesize$F_m$};
    \end{tikzpicture}
\end{center}
Then $(X\ni x,E)$ is plt.
\end{lem}

\begin{proof}
We may write
$$K_Y+E_Y-\sum_{i=1}^ma_iF_i=f^*(K_X+E)$$
where $a_i:=a(F_i,X,E)-1$. Let $H:=-\sum_{i=1}^mF_i$ and $A:=\sum_{i=1}^ma_iF_i$. Then
$$H\cdot F_1=-F_1^2=2+K_Y\cdot F_1=1+(K_Y+E_Y)\cdot F_1=1+A\cdot F_1>A\cdot F_1$$
if $m=1$,
$$H\cdot F_1=-1-F_1^2=1+K_Y\cdot F_1=(K_Y+E_Y)\cdot F_1=A\cdot F_1$$
if $m\geq 2$,
$$H\cdot F_i=-2-F_i^2=K_Y\cdot F_i=(K_Y+E_Y)\cdot F_i=A\cdot F_i$$
if $m\geq 2$ and $2\leq i\leq m-1$, and
$$H\cdot F_m=-1-F_m^2=1+K_Y\cdot F_m=1+(K_Y+E_Y)\cdot F_m=1+A\cdot F_i>A\cdot F_m$$
if $m\geq 2$. By Lemma \ref{lem: intersection number compare coefficients}, $a_i>-1$ for each $i$, hence $(X\ni x,E)$ is plt.
\end{proof}

\begin{thm}\label{thm: a type kc}
Let $(X\ni x,B)$ be a plt surface germ such that $X\ni x$ is an $A$-type singularity. Then for any divisor $E$ over $X\ni x$ such that $a(E,X,B)=\mld(X\ni x,B)$, $E$ is a Koll\'ar component of $X\ni x$.
\end{thm}
\begin{proof}
For any model $X'$ of $X$ such that $\Center_{X'}E$ is a divisor, we let $E_{X'}$ be the center of $E$ on $X'$. Let $h: Y\rightarrow X$ be the minimal resolution of $X$ and $F_1,\dots,F_m$ the prime exceptional divisors of $h$ with the following dual graph:
\begin{center}
    \begin{tikzpicture}
                                                    
         \draw (1.5,0) circle (0.1);
         \node [above] at (1.5,0.1) {\footnotesize$F_1$};
         \draw (1.6,0)--(2,0);
         \draw (2.1,0) circle (0.1);
         \node [above] at (2.1,0.1) {\footnotesize$F_2$};
         \draw[dashed] (2.2,0)--(3.2,0);
         \draw (3.3,0) circle (0.1);
         \node [above] at (3.3,0.1) {\footnotesize$F_{m-1}$};
         \draw (3.4,0)--(3.8,0);

         \draw (3.9,0) circle (0.1);
         
         
         \node [above] at (4.0,0.1) {\footnotesize$F_m$};
    \end{tikzpicture}
\end{center}
Let $B_Y:=f^{-1}_*B$ and $a_i:=1-a(F_i,X,B)$ for each $i$. Then $a_i\in [0,1)$ for each $i$. There are three cases:

\medskip

\noindent\textbf{Case 1}. $E$ is on $Y$. In this case, we let $W:=Y$ and $g:=h$. Then $\Dd(g)$ is a chain.  Moreover, by the construction of $g$, for any prime divisor $F\not=E_W$ in $\Exc(g)$, $F^2\leq -2$.

\medskip

\noindent\textbf{Case 2}. $E$ is not on $Y$, $\Center_YE:=y\in F_i$ for some $i$, and $\Center_YE\not\in F_j$ for any $j\not=i$. 
In this case, since $a(E,X,B)=\mld(X\ni x,B)$, $a(E,X,B)\leq 1-a_i$. By Lemma \ref{lem: interseciton<1 imply mld does not need blow up}, $B_Y\cdot F_i\geq 1$. By Lemma \ref{lem: km98 4.15 generalization}(3), $i=1$ or $m$. We let $f_E: W\rightarrow Y$ the sequence of smooth blow-ups as in Definition-Lemma \ref{deflem: fE}, and $g:=h\circ f_E$. By Lemma \ref{lem: km98 4.15 generalization}(1), $B_Y\cdot F_i<2$, hence $(B_Y\cdot F_i)_y<2$. Since
$$K_Y+B_Y+\sum a_iF_i=h^*(K_X+B),$$
by Lemma \ref{lem: blow-up not intersection is chain}, $\Dd(g)$ is a chain. Moreover, by the construction of $g$, for any prime divisor $F\not=E_W$ in $\Exc(g)$, $F^2\leq -2$.

\medskip

\noindent\textbf{Case 3}. $E$ is not on $Y$ and $\Center_YE:=y\in F_i\cap F_{i+1}$ for some $i$. In this case, we let $f_E: W\rightarrow Y$ the sequence of smooth blow-ups as in Definition-Lemma \ref{deflem: fE}, and $g:=h\circ f_E$. By Lemma \ref{lem: km98 4.15 generalization}(2), either $B_Y\cdot F_i<1$ or $B_Y\cdot F_{i+1}<1$, hence  either $(B_Y\cdot F_i)_y<1$ or $(B_Y\cdot F_{i+1})_y<1$. Since
$$K_Y+B_Y+\sum a_iF_i=h^*(K_X+B),$$
by Lemma \ref{lem: blow-up not intersection is chain two tails}, $\Dd(g)$ is a chain. Moreover, by the construction of $g$, for any prime divisor $F\not=E_W$ in $\Exc(g)$, $F^2\leq -2$.

\medskip

There exists a contraction $\phi: W\rightarrow Z$ over $X$ of $\Supp\Exc(g)\backslash E_W$. By our construction and Lemma \ref{lem: a type is plt}, in any case above,  $(Z\ni z,E_Z)$ is plt for any singular point $z$ of $Z$ in $E_Z$. Thus $(Z,E_Z)$ is plt near $E_Z$, hence $E$ is a Koll\'ar component of $X\ni x$.
\end{proof}

\subsection{\texorpdfstring{$D$}{}-type and \texorpdfstring{$E$}{}-type singularities}

\begin{lem}\label{lem: d type is lc not plt}
Let $X\ni x$ be a surface germ of $D_m$-type for some integer $m\geq 3$, $E$ a prime divisor on $X$, $f: Y\rightarrow X$ the minimal resolution of $X\ni x$, and $E_Y:=f^{-1}_*E$. Let $F_1,\dots,F_m$ be the prime exceptional divisors over $X\ni x$. Assume that $E_Y\cup_{i=1}^mF_i$ is simple normal crossing over a neighborhood of $x$, $F_{m-1}^2=F_m^2=-2$, and the dual graph of $E_Y\cup_{i=1}^mF_i$ is the following:
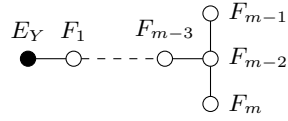
\begin{center}
    \begin{tikzpicture}
                                                    
         \draw[fill=black] (1.5,0) circle (0.1);
         \node [above] at (1.5,0.1) {\footnotesize$E_Y$};
         \draw (1.6,0)--(2,0);
         \draw (2.1,0) circle (0.1);
         \node [above] at (2.1,0.1) {\footnotesize$F_{1}$};
         \draw[dashed] (2.2,0)--(3.2,0);
         \draw (3.3,0) circle (0.1);
         \node [above] at (3.3,0.1) {\footnotesize$F_{m-3}$};
         \draw (3.4,0)--(3.8,0);

         \draw (3.9,0) circle (0.1);
         \draw (3.9,0.6) circle (0.1);
         \draw (3.9,-0.6) circle (0.1);
         
         \draw (3.9,0.1)--(3.9,0.5);
         \draw (3.9,-0.1)--(3.9,-0.5);
         
         \node [right] at (4,0) {\footnotesize$F_{m-2}$};
         \node [right] at (4,0.6) {\footnotesize$F_{m-1}$};
         \node [right] at (4,-0.6) {\footnotesize$F_{m}$};
    \end{tikzpicture}
\end{center}
Then $(X\ni x,E)$ is lc but not plt.
\end{lem}
\begin{proof}
By computing intersections numbers, 
$$K_Y+E_Y+\sum_{i=1}^{m-2}F_i+\frac{1}{2}(F_{m-1}+F_m)=f^*(K_X+E).$$
Since $(Y,E_Y+\sum_{i=1}^{m-2}F_i+\frac{1}{2}(F_{m-1}+F_m))$ is log smooth over a neighborhood of $x$, $(X\ni x,E)$ is lc but not plt.
\end{proof}

\begin{lem}\label{lem: Discrepancy of the fork is minimal}
Let $0<m_1<m_2<m_3:=m$ be integers, $(X\ni x,B)$ a plt surface germ, $f:Y\to X$ the minimal resolution of $X\ni x$ with prime exceptional divisors $F_0,\dots,F_{m}$ with the following dual graph $\Dd(f)$:
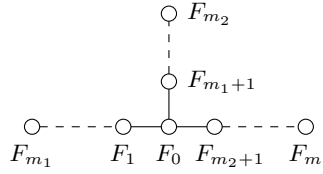
\begin{center}
    \begin{tikzpicture}
                                                    
         \draw (2.1,0) circle (0.1);
         \node [below] at (2.1,-0.1) {\footnotesize$F_{m_1}$};
         \draw[dashed] (2.2,0)--(3.2,0);
         \draw (3.3,0) circle (0.1);
         \node [below] at (3.3,-0.1) {\footnotesize$F_1$};
         \draw (3.4,0)--(3.8,0);

         \draw (3.9,+1.5) circle (0.1);
         \draw (3.9,0) circle (0.1);
         \draw (3.9,0.6) circle (0.1);
         \draw (4.5,0) circle (0.1);
         \node [below] at (4.7,-0.1) {\footnotesize$F_{m_2+1}$};
         \draw (5.7,0) circle (0.1);
         \node [below] at (5.7,-0.1) {\footnotesize$F_m$};

         \draw [dashed](3.9,+0.7)--(3.9,+1.4);
         \draw (3.9,0.1)--(3.9,0.5);
         \draw (4.0,0)--(4.4,0);
         \draw [dashed](4.6,0)--(5.6,0);

         \node [right] at (4,1.5) {\footnotesize$F_{m_2}$};
         \node [below] at (3.9,-0.1) {\footnotesize$F_0$};
         \node [right] at (4,0.6) {\footnotesize$F_{m_1+1}$};
    \end{tikzpicture}
\end{center}
Let $a_i:=a(F_i,X,B)-1$ for each $i$. Then
\begin{enumerate}
    \item $a(F_0,X,B)\leq a(F_i,X,B)$ for any $i\in\{1,2,\dots,m_1\}$, and
    \item if $a(F_0,X,B)=a(F_l,X,B)$ for some $l\in\{1,2,\dots,m_1\}$, then
    \begin{enumerate}
    \item $a_i=a_0$ for every $i\in\{0,1,\dots,l\}$, 
    \item $a_{m_1+1}=a_{m_2+1}=\frac{1}{2}a_0$,
    \item $F_i^2=-2$ for any $i\in\{0,1,\dots,l-1,m_1+1,m_2+1\}$, and
    \item either $m=m_2+1=m_1+2$, or $B=0$ and $X\ni x$ Du Val.
    \end{enumerate}
\end{enumerate}
\end{lem}
\begin{proof}
Let $B_Y:=f^{-1}_*B$. Then 
$$K_Y+B_Y-\sum_{i=1}^ma_iF_i=f^*(K_X+B).$$
Since $(X\ni x,B)$ is plt and $f$ is the minimal resolution of $X\ni x$, $-1<a_i\leq 0$ for each $i$.

If $a_0<a_i$ for any $i\in\{1,2,\dots,m_1\}$ there is nothing to prove. Otherwise, there exists $k\in\{1,2,\dots,m_1\}$, such that $a_k=\min\{a_i\mid 0\leq i\leq m_1\}\leq a_0$. We define 
$$A:=\sum_{i=0}^{k}a_iF_i+a_{m_1+1}F_{m_1+1}+a_{m_2+1}F_{m_2+1}, \text{ and } H:=a_k(\sum_{i=0}^kF_i+\frac{1}{2}F_{m_1+1}+\frac{1}{2}F_{m_2+1}).$$
Then
$$H\cdot F_i=a_k(F_i^2+2)=-a_kK_Y\cdot F_i\leq K_Y\cdot F_i\leq(K_Y+B_Y)\cdot F_i=A\cdot F_i,$$
when $0\leq i<k$ and if the equality holds then $K_Y\cdot F_i=0$,
$$H\cdot F_{k}=a_k(F_k^2+1)\leq a_kF_k^2+a_{k-1}=A\cdot F_k$$
and if the equality holds then $a_k=a_{k-1}$, and  
$$H\cdot F_{m_i+1}=a_k(1+\frac{1}{2}F_{m_i+1}^2)=-\frac{a_k}{2}K_Y\cdot F_{m_i+1}\leq K_Y\cdot F_{m_i+1}\leq (K_Y+B_Y)\cdot F_{m_i+1}\leq A\cdot F_{m_i+1}$$
for $i\in\{1,2\}$, and if the equality holds, then
\begin{itemize}
    \item $K_Y\cdot F_{m_i+1}=0$, and
    \item either $m_{i+1}=m_i+1$, or $m_{i+1}\geq m_i+2$ and $a_{m_i+2}=0$.
\end{itemize}
Thus $H\cdot F_i\leq A\cdot F_i$ for any $i\in\{0,1,\dots,k,m_1+1,m_2+1\}$. By Lemma \ref{lem: intersection number compare coefficients}, $a_i\leq a_k$ for any $i\in\{0,1,\dots,k\}$ and $a_{m_1+1}=a_{m_2+1}=\frac{1}{2}a_k$. Since $a_k=\min\{a_i\mid 0\leq i\leq m_1\}\leq a_0$, $a_i=a_k=\min\{a_i\mid 0\leq i\leq m_1\}$ for any $i\in\{0,1,\dots,k\}$. Thus for any $l\in\{1,2,\dots,m_1\}$ such that $a(F_0,X,B)=a(F_l,X,B)$, we may pick $k=l$, which shows (1) and (2.a). Moreover, we have that $H\cdot F_i=A\cdot F_i$ for any $i\in\{0,1,\dots,l,m_1+1,m_2+1\}$, which implies that
\begin{itemize}
\item $a_{m_1+1}=a_{m_2+1}=\frac{1}{2}a_l$, hence (2.b).
\item $F_i^2=K_Y\cdot F_i=-2$ for every $i\in\{0,1,\dots,l-1,m_1+1,m_2+1\}$, hence (2.c), and
\item either $m=m_2+1=m_1+2$, or there exists $i\in\{1,2\}$ such that $m_{i+1}\geq m_i+2$ and $a_{m_i+2}=0$.
\end{itemize}
If $m=m_2+1=m_1+2$ then we get (2.d) and the proof is completed. Otherwise, there exists $i\in\{1,2\}$ such that $m_{i+1}\geq m_i+2$ and $a_{m_i+2}=0$. Thus
$$1=a(F_{m_i+2},X,B)\leq a(F_{m_i+2},X,0)\leq 1,$$
which implies that $a(F_i,X,B)=a(F_i,X,0)$, hence $B=0$. Moreover, since $f$ is the minimal resolution of $X\ni x$, $\sum_{i=1}^ma_iF_i\sim_XK_Y$ is nef over $X$. By Lemma \ref{lem: intersection number compare coefficients}, $a_i=0$ for every $i$, and $X\ni x$ is Du Val.
\end{proof}

\begin{lem}\label{lem: de type mld reached at minimal resolution}
Let $(X\ni x,B)$ be a plt surface germ such that $X\ni x$ is a $D_m$-type singularity for some integer $m\geq 4$, or an $E_m$-type singularity for some integer $m\in\{6,7,8\}$. Let $f: Y\rightarrow X$ be the minimal resolution of $X\ni x$, and $E$ a divisor over $X\ni x$ such that $a(E,X,B)=\mld(X\ni x,B)$. Then $E\subset\Exc(f)$.
\end{lem}
\begin{proof}
Let $F_1,\dots,F_m$ be the prime exceptional divisors of $f$, $a_i:=1-a(F_i,X,B)$ for each $i$, and $B_Y:=f^{-1}_*B$. Then
$$K_Y+B_Y+\sum_{j=1}^ma_jF_j=f^*(K_X+B).$$
Suppose that $E\not\subset\Exc(f)$. Let $y:=\Center_YE$.

\begin{claim}\label{claim: de type mld center only intersection}
$y=F_i\cap F_k$ for some $i\not=k$.
\end{claim}
\begin{proof}
Then there exists an integer $i\in\{1,2,\dots,m\}$ such that $y\in F_i$. Thus
\begin{align*}
    \mld(Y\ni y, B_Y+\sum_{j=1}^ma_jF_j)&\leq a(E,Y, B_Y+\sum_{j=1}^ma_jF_j)\\
    &=a(E,X,B)=\mld(X\ni x,B)\leq 1-a_i.
\end{align*}
By Lemma \ref{lem: interseciton<1 imply mld does not need blow up}, 
$$(B+\sum_{j\not=i}a_jF_j)\cdot F_i\geq((B+\sum_{j\not=i}a_jF_j)\cdot F_i)_y\geq 1.$$
By Lemma \ref{lem: km98 4.15 generalization}(3), $B\cdot F_i<1$, which implies that there exists $k\not=i$ such that $y\in F_k$.  In particular, $y=F_i\cap F_k$.
\end{proof}
\noindent\textit{Proof of Lemma \ref{lem: de type mld reached at minimal resolution}} continued. By Claim \ref{claim: de type mld center only intersection}, $y=F_i\cap F_k$ for some $i\not=k$. Possibly switching $i$ and $k$, we may assume that $F_i$ is closer to the fork of $\Dd(f)$ than $F_k$. Since $(X\ni x,B)$ is plt and $f$ is the minimal resolution of $X\ni x$, $0\leq a_k<1$. Thus there exists an extraction $g: W\rightarrow X$ of $F_k$ with induced morphism $h: Y\rightarrow W$. Let $\bar F_k$ be the center of $F_k$ on $W$ and $w:=\Center_WE$. Then $h$ is the minimal resolution of $W\ni w$. Moreover, if $F_i$ is not the fork of $\Dd(f)$, then $W\ni w$ is not an $A$-type singularity, and if $F_i$ is the fork of $\Dd(f)$, then $W\ni w$ is an $A$-type singularity but $F_i$ is not a tail of $\Dd(h)$. Since $(X\ni x,B)$ is plt, $(W\ni w,g^{-1}_*B+a_kF_k)$ is plt. By Lemma \ref{lem: km98 4.15 generalization}(3), 
$$((B+\sum_{j\not=i}a_jF_j)\cdot F_i)_y=((B+a_kF_k)\cdot F_i)_y\leq (B+a_kF_k)\cdot F_i<1.$$
By Lemma \ref{lem: interseciton<1 imply mld does not need blow up}, 
\begin{align*}
    \mld(X\ni x,B)&=a(E,X,B)=a(E,Y, B_Y+\sum_{j=1}^ma_jF_j)\\
    &\geq\mld(Y\ni y, B_Y+\sum_{j=1}^ma_jF_j)>1-a_i\geq\mld(X\ni x,B),
\end{align*}
a contradiction.
\end{proof}

\begin{thm}\label{thm: de type kc}
Let $(X\ni x,B)$ be a plt surface germ such that $X\ni x$ is a $D_m$-type singularity for some integer $m\geq 4$ or an $E_m$-type singularity for some integer $m\in\{6,7,8\}$. Let $f: Y\rightarrow X$ be the minimal resolution of $X\ni x$. Then there exists a unique divisor $E$ over $X\ni x$, such that $a(E,X,B)=\mld(X\ni x,B)$, and $E$ is a Koll\'ar component of $X\ni x$. Moreover, $E$ is the unique fork of $\Dd(f)$.
\end{thm}
\begin{proof}
By Lemma \ref{lem: de type mld reached at minimal resolution}, we may only consider divisors in $\Exc(f)$. Let $F_1,\dots,F_m$ be the prime exceptional divisors of $f$ such that $F_1$ is the unique fork. Since $a(F_i,X,B)\leq a(F_i,X,0)\leq 1$ for each $i$, there exists an extraction $g: Y_i\rightarrow X$ of $F_i$ for each $i$, and we let $\bar F_i$ be the strict transform of $F_i$ on $Y_i$. By Lemmas \ref{lem: a type is plt} and \ref{lem: d type is lc not plt}, $(Y_i,\bar F_i)$ is not plt near $\bar F_i$ for any $i\in\{2,3,\dots,m\}$ and $(Y_1,\bar F_1)$ is plt near $F_1$. By Lemma \ref{lem: Discrepancy of the fork is minimal}, $a(F_1,X,B)=\mld(X\ni x,B)$. So $E=F_1$ is the unique divisor we want.
\end{proof}

\begin{thm}\label{thm: de type potential lc place not du val}
Let $(X\ni x,B)$ be a plt surface germ such that $X\ni x$ is a $D_m$-type singularity for some integer $m\geq 4$ or an $E_m$-type singularity for some integer $m\in\{6,7,8\}$. Assume that either $B\not=0$ or $X\ni x$ is not Du Val. Let $E$ be a divisor over $X\ni x$ such that $a(E,X,B)=\mld(X\ni x,B)$, then $E$ is a potential lc place of $X\ni x$. 
\end{thm}
\begin{proof}
Let $f: Y\rightarrow X$ be the minimal resolution of $X\ni x$. By Lemma \ref{lem: de type mld reached at minimal resolution}, $E\subset\Exc(f)$. By Theorem \ref{thm: de type kc}, we may assume that $E$ is not the fork of $\Dd(f)$. Let $L_1,L_2,L_3$ be the three branches of $\Dd(f)$ and assume that $E$ belongs to $L_1$. By Lemma \ref{lem: Discrepancy of the fork is minimal}(2.c,2.d), $X\ni x$ is a $D_m$-type singularity, $L_2$ contains a unique curve $F_2$ and $L_3$ contains a unique curve $F_3$ respectively, such that $F_2^2=F_3^2=-2$. Since $(X\ni x,B)$ is plt and $f$ is the minimal resolution of $X\ni x$, $0<a(E,X,B)\leq 1$, so there exists an extraction $g: W\rightarrow X$ of $E$ with the induced morphism $h: Y\rightarrow W$. Let $E_W$ be the strict transform of $E$ on $W$. By Lemmas \ref{lem: a type is plt} and \ref{lem: d type is lc not plt}, $(W,E_W)$ is lc near $E_W$, so $E$ is a potential lc place of $X\ni x$.
\end{proof}

\begin{thm}\label{thm: de type potential lc place du val}
Let $X\ni x$ be a Du Val singularity of $D_m$-type for some integer $m\geq 4$ or of $E_m$-type for some integer $m\in\{6,7,8\}$. Let $f: Y\rightarrow X$ be the minimal resolution of $X\ni x$, and $E$ be a divisor over $X\ni x$ such that $a(E,X,0)=\mld(X\ni x,0)$, then $E$ is a potential lc place of $X\ni x$ if and only if one of the following holds:
\begin{enumerate}
    \item $E$ is the fork of the $\Dd(f)$.
    \item $X\ni x$ is of $D_m$-type, the dual graph of $X\ni x$ looks like the following, and $E\in\{F_1,\dots,F_{m-2}\}$.
\end{enumerate}
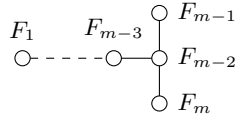
\begin{center}
    \begin{tikzpicture}
                                                    
         \draw (2.1,0) circle (0.1);
         \node [above] at (2.1,0.1) {\footnotesize$F_{1}$};
         \draw[dashed] (2.2,0)--(3.2,0);
         \draw (3.3,0) circle (0.1);
         \node [above] at (3.3,0.1) {\footnotesize$F_{m-3}$};
         \draw (3.4,0)--(3.8,0);

         \draw (3.9,0) circle (0.1);
         \draw (3.9,0.6) circle (0.1);
         \draw (3.9,-0.6) circle (0.1);
         
         \draw (3.9,0.1)--(3.9,0.5);
         \draw (3.9,-0.1)--(3.9,-0.5);
         
         \node [right] at (4,0) {\footnotesize$F_{m-2}$};
         \node [right] at (4,0.6) {\footnotesize$F_{m-1}$};
         \node [right] at (4,-0.6) {\footnotesize$F_{m}$};
    \end{tikzpicture}
\end{center}
\end{thm}
\begin{proof}
Since $a(E,X,0)=\mld(X\ni x,0)$, $E\subset\Exc(f)$. Since $X\ni x$ is Du Val, for any prime divisor $F\subset\Exc(f)$, $F^2=-2$. The if part follows from Lemmas \ref{lem: a type is plt} and \ref{lem: d type is lc not plt}.

To prove the only if part, we may assume that $E$ is not the fork of $\mathcal{D}(f)$. Then $E$ is a potential lc place of $X\ni x$ if and only if there exists an extraction $g: W\rightarrow X$ of $E$ such that $(W,E_W)$ is lc near $E_W$, where $E_W$ is the strict transform of $E$ on $W$. The theorem follows from \cite[Theorem 4.15]{KM98}.
\end{proof}

\subsection{Non-plt singularities}

\begin{deflem}[{cf. \cite[Theorem 4.7]{KM98}}]\label{deflem: lc surface sing}
Let $X\ni x$ be an lc but not klt surface germ and $f: Y\rightarrow X$ the minimal resolution of $X\ni x$. Then exactly one of the following holds:
\begin{enumerate}
    \item ({\bf B}-type) $\Exc(f)=F$ is a smooth elliptic curve.
    \item ({\bf C}-type) $\Exc(f)=F$ is a nodal cubic curve.
    \item ({\bf F}-type) $\Exc(f)$ is a circle of smooth rational curves.
    \item ({\bf H}-type) $\Exc(f)$ has $\geq 5$ rational curves and $\Dd(f)$ has the following weighted dual graph:
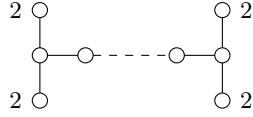
\begin{center}
    \begin{tikzpicture}
                                                    
         \draw (1.5,0.6) circle (0.1);
         \node [left] at (1.4,0.6) {\footnotesize$2$};
         \draw (1.5,0.1)--(1.5,0.5);
         \draw (1.5,0) circle (0.1);
         \draw (1.5,-0.6) circle (0.1);
         \node [left] at (1.4,-0.6) {\footnotesize$2$};
         \draw (1.5,-0.1)--(1.5,-0.5);
         \draw (1.6,0)--(2,0);
         \draw (2.1,0) circle (0.1);
         \draw[dashed] (2.2,0)--(3.2,0);
         \draw (3.3,0) circle (0.1);
         \draw (3.4,0)--(3.8,0);

         \draw (3.9,0) circle (0.1);
         \draw (3.9,0.6) circle (0.1);
         \draw (3.9,-0.6) circle (0.1);
         
         \draw (3.9,0.1)--(3.9,0.5);
         \draw (3.9,-0.1)--(3.9,-0.5);
         
         \node [right] at (4,0.6) {\footnotesize$2$};
         \node [right] at (4,-0.6) {\footnotesize$2$};
    \end{tikzpicture}
\end{center}
\end{enumerate}
Let $m\geq 5$ be an integer. For an lc singularity of {\bf H}-type as above with $m$ exceptional divisors, we call it an \emph{${\bf{H}}_m$-type singularity}.
\end{deflem}

Although the concept of Koll\'ar component is not defined over an lc but not klt germ, the classification of surface lc singularities tells us when there exists a divisor which ``looks like a Koll\'ar component".
\begin{thm}\label{thm: lc type kc}
Let $(X\ni x,B)$ be an lc surface germ such that $X\ni x$ is not klt, and $E$ a prime divisor over $X$ such that $a(E,X,B)=\mld(X\ni x,B)$. Then
\begin{enumerate}
    \item $E$ is a potential lc place of $(X\ni x,B)$. 
    \item $K_Y+E$ is plt near $E$ if and only if $E$ is a {\bf B}-type or an ${\bf{H}}_5$-type singularity.
\end{enumerate}
\end{thm}
\begin{proof}
$E$ is an lc place of $(X\ni x,B)$ which implies (1). Since $(X\ni x,B)$ is lc and $X\ni x$ is not klt, $B=0$. By the connectedness of lc places, $K_Y+E$ is plt near $E$ if and only if $E$ is the only lc place over $X\ni x$, and (2) follows from Definition-Lemma \ref{deflem: lc surface sing}.
\end{proof}

Now we deal with the case when $X\ni x$ is klt but $(X\ni x,B)$ is not plt:

\begin{thm}\label{thm: lc not plt x klt potential lc place}
Let $(X\ni x,B)$ be an lc surface germ such that $X\ni x$ is klt but $(X\ni x,B)$ is not plt. Then any divisor $E$ over $X\ni x$ such that $a(E,X,B)=\mld(X\ni x,B)$ is a potential lc place of $(X\ni x,B)$.
\end{thm}
\begin{proof}
Since $(X\ni x,B)$ is not plt, $a(E,X,B)=\mld(X\ni x,B)=0$, so $E$ is an lc place of $(X\ni x,B)$, hence a potential lc place of $(X\ni x,B)$.
\end{proof}

\begin{thm}\label{thm: lc not klt kc}
Let $(X\ni x,B)$ be an lc surface germ such that $X\ni x$ is klt but $(X\ni x,B)$ is not plt. Then there exists a divisor $E$ over $X\ni x$ such that $a(E,X,B)=\mld(X\ni x,B)=0$ and $E$ is Koll\'ar component of $X\ni x$.
\end{thm}
\begin{proof}
Let $$\delta:=\min\{a(E,X,B)\mid E\text{ is over }X\ni x, a(E,X,B)>0\}.$$
Then there exists $\epsilon\in (0,1)$ such that $\delta>\mld(X\ni x,(1-\epsilon)B)>0$. By Theorems \ref{thm: a type kc} and \ref{thm: de type kc}, there exists a divisor $E$ over $X\ni x$ such that $a(E,X,(1-\epsilon)B)=\mld(X\ni x,(1-\epsilon)B)<\delta$ and $E$ is a Koll\'ar component of $X\ni x$. Since 
$$0\leq a(E,X,B)<a(E,X,(1-\epsilon)B)<\delta,$$
$a(E,X,B)=\mld(X\ni x,B)=0$, so $E$ satisfies our requirements.
\end{proof}

Finally, recall the following result:
\begin{thm}\label{thm: surface dlt is log smooth}
Let $(X\ni x,B)$ be a dlt surface germ that is not plt. Then $(X\ni x,B)$ is log smooth and $B=\lfloor B\rfloor$ has exactly two irreducible components near $x$.
\end{thm}
\begin{proof}
Since $(X\ni x,B)$ is dlt but not plt, $\lfloor B\rfloor$ contains at least $2$ irreducible components near $x$. Since $X$ is a surface, $B=\lfloor B\rfloor$ has exactly two irreducible components near $x$.

There exists a divisor $E$ over $X\ni x$ such that $a(E,X,B)=0$. Since $(X\ni x,B)$ is dlt, $x:=\Center_XE$ belongs to the log smooth strata of $(X,B)$, so $(X\ni x,B)$ is log smooth. 
\end{proof}

\section{Proof of the main theorems}

\begin{proof}[Proof of Theorem \ref{thm: classification divisors computing mld on surfaces}]
(1.a) follows from Theorems \ref{thm: Kaw17 1.1}, \ref{thm: a type kc}, and \ref{thm: surface dlt is log smooth}. For (1.b), by Theorem \ref{thm: surface dlt is log smooth}, $(X\ni x,B)$ is plt. (1.b.i) follows from Theorem \ref{thm: de type kc}. (1.b.ii) follows from Lemma \ref{lem: de type mld reached at minimal resolution}. (1.b.iii.A) follows from Theorem \ref{thm: de type potential lc place not du val}. (1.b.iii.B) follows from Lemma \ref{lem: Discrepancy of the fork is minimal} and (1.b.ii). (1.b.iv.A) is the classification of surface Du Val singularities. (1.b.iv.B) and (1.b.iv.C) follow from Theorem \ref{thm: de type potential lc place du val}. (2.a) follows from Theorem \ref{thm: lc not plt x klt potential lc place}. (2.b) follows from Theorem \ref{thm: lc not klt kc}. (2.c) follows from Theorem \ref{thm: Kaw17 1.1}. (3) follows from Theorem \ref{thm: lc type kc}.
\end{proof}

\begin{proof}[Proof of Theorem \ref{thm: dim 2 compute mld is kc}]
It follows from Theorem \ref{thm: classification divisors computing mld on surfaces}(1.a, 1.b.i, 2.b).
\end{proof}

\begin{proof}[Proof of Theorem \ref{thm: when all divisors computing mld are potential lc places}]
It follows from Theorem \ref{thm: classification divisors computing mld on surfaces}(1.a, 1.b.iii.A, 1.b.iv.B, i.b.iv.C, 2.a, 3).
\end{proof}

\section{Examples}

The following example is given by Zhuang which shows that Question \ref{ques: div compute mld is kc} does not have a general positive answer in dimension $\geq 3$ even when $B=0$. We are grateful for him sharing the example with us.

\begin{ex}[{c.f. \cite[Excercise 41]{Kol08}}]\label{ex: threefold mld not kc}
Consider the threefold singularity given by 
$$(x^3+y^3+z^3+w^4=0)\subset\mathbb (\mathbb C^4\ni 0).$$
This is a canonical singularity, and the only divisor $E$ which computes the mld is attained at the ordinary blow-up. However, $E$ is a cone over an elliptic curve, so $E$ is lc but not klt. In particular, $E$ is not a Koll\'ar component of the ambient variety.
\end{ex}

The following example of Kawakita shows that there may not exist a divisor computing $\mld(X\ni x,B)$ that is also a potential place of $(X\ni x,B)$ even when $X$ is a smooth surface. We remark that Theorem \ref{thm: dim 2 compute mld is kc} shows that there always exists a divisor computing $\mld(X\ni x,B)$ that is a Koll\'ar component of $(X\ni x,0)$.

\begin{ex}[{cf. \cite[Example 2]{Kaw17}}]
Let $D:=(x_1^2+x_2^3+rx_1x_2^2=0)\subset\mathbb A^2$ for some general real number $r$, and $B:=\frac{2}{3}D$. Then there exists a unique divisor $E$ over $\mathbb A^2\ni 0$ such that $\mld(\mathbb A^2\ni 0,B)=a(E,X,B)=\frac{2}{3}$. However, $E$ is not a potential place of $(\mathbb A^2\ni 0,B)$. 
\end{ex}

The following example is complementary to Theorem \ref{thm: a type kc}, which shows that the assumption ``$(X\ni x,B)$ is plt" is necessary.

\begin{ex}\label{ex: not plt a type may not be kc}
Let $Z\ni z$ be a $D_4$-type Du Val singularity with the following dual graph:
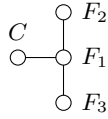
\begin{center}
    \begin{tikzpicture}
                                                    
         \draw (3.3,0) circle (0.1);
         \node [above] at (3.3,0.1) {\footnotesize$C$};
         \draw (3.4,0)--(3.8,0);

         \draw (3.9,0) circle (0.1);
         \draw (3.9,0.6) circle (0.1);
         \draw (3.9,-0.6) circle (0.1);
         
         \draw (3.9,0.1)--(3.9,0.5);
         \draw (3.9,-0.1)--(3.9,-0.5);
         
         \node [right] at (4,0) {\footnotesize$F_1$};
         \node [right] at (4,0.6) {\footnotesize$F_2$};
         \node [right] at (4,-0.6) {\footnotesize$F_3$};
    \end{tikzpicture}
\end{center}
Let $g: X\rightarrow Z$ be the extraction of $C$. Then $X$ has a unique singularity $x\in C$ such that $X\ni x$ is an $A_3$-type Du Val singularity. Let $f: Y\rightarrow X$ be the minimal resolution of $(X\ni x,C)$ and $C_Y:=f^{-1}_*C$. Then 
$$K_Y+C_Y+F_1+\frac{1}{2}F_2+\frac{1}{2}F_3=f^*(K_X+C).$$
In particular, let $h: W\rightarrow Y$ be the smooth blow-up of $C_Y\cap F_1$ with the exceptional divisor $E$. Then $a(E,X,C)=\mld(X\ni x,C)=0$, but by \cite[Theorem 4.15(2)]{KM98}, $E$ is not a Koll\'ar component of $X\ni x$.
\end{ex}

The last example is complementary to Theorem \ref{thm: de type kc} and \ref{thm: de type potential lc place du val}, which shows that even for non-Du Val singularities of $D$-type and $B=0$, it is possible that some divisor which computes the minimal log discrepancy is not a Koll\'ar component.
\begin{ex}\label{ex: not du val d type may not be kc}
Let $G:=\text{BD}_{12}(5,3)$ be a binary dihedral group in $\text{\rm GL}(2,\mathbb C)$. By \cite[Table 3.2]{dC08}, the quotient singularity $X\ni x\cong\mathbb C^2/G\ni 0$ is a $D_4$-type singularity and its minimal resolution $f: Y\rightarrow X$ has the following dual graph:
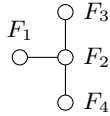
\begin{center} 
    \begin{tikzpicture}
                                                    
         \draw (3.3,0) circle (0.1);
         \node [above] at (3.3,0.1) {\footnotesize$F_1$};
         \draw (3.4,0)--(3.8,0);

         \draw (3.9,0) circle (0.1);
         \draw (3.9,0.6) circle (0.1);
         \draw (3.9,-0.6) circle (0.1);
         
         \draw (3.9,0.1)--(3.9,0.5);
         \draw (3.9,-0.1)--(3.9,-0.5);

         \node [right] at (4,0) {\footnotesize$F_2$};
         \node [right] at (4,0.6) {\footnotesize$F_3$};
         \node [right] at (4,-0.6) {\footnotesize$F_4$};

    \end{tikzpicture}
\end{center}
where $F_1^2=-3$ and $F_i^2=-2$ for $i\in\{2,3,4\}$. THus $a(F_1,X,0)=\frac{1}{2}=\mld(X\ni x,0)$, but by Theorem \ref{thm: de type kc}, $F_1$ is not a Koll\'ar component of $X\ni x$.
\end{ex}

\end{document}